\documentclass[12pt]{amsart}

\usepackage{amsmath,amstext,amssymb,amsopn,amsthm}
\usepackage{url,verbatim}
\usepackage{mathtools}
\usepackage{enumerate}

\usepackage{color,graphicx}

\usepackage[margin=30mm]{geometry}
\usepackage{eucal,mathrsfs,dsfont}

\numberwithin{equation}{section}

\allowdisplaybreaks

\usepackage{graphicx,amsmath,amssymb,enumerate,color,amsthm}
\usepackage{hyperref}

\allowdisplaybreaks

\newtheorem{theorem}{Theorem}[section]

\theoremstyle{definition}

\newtheorem{remark}[theorem]{Remark}

\numberwithin{equation}{section}

\newcommand{\eps}{\varepsilon}

\newcommand{\calF}{\mathcal{F}}

\newcommand{\calG}{\mathcal{G}}

\newcommand{\calS}{\mathcal{S}}

\newcommand{\calB}{\mathcal{B}}

\renewcommand{\P}{\operatorname{\mathds{P}}} 
\newcommand{\R}{\mathds{R}}
\newcommand{\B}{\calB}
\newcommand{\bw}{\mathbf{w}}
\newcommand{\C}{\mathds{C}}

\newcommand{\prt}{\partial}

\newcommand{\ol}{\overline}
\newcommand{\wh}{\widehat}
\newcommand{\wt}{\widetilde}

\DeclareMathOperator{\dist}{dist}

\def\Re{{\rm Re\,}}
\def\Im{{\rm Im\,}}

\def\bone{{\bf 1}}
\def\n{{\bf n}}

\def\bv{{\bf v}}
\def\bz{{\bf z}}

\title[Brownian couplings]{Simultaneous boundary hitting \\ by coupled reflected Brownian motions}
\author{Krzysztof Burdzy }

\address{Department of Mathematics, Box 354350, University of Washington, Seattle, WA 98195}
\email{burdzy@uw.edu}

\thanks{Research supported in part by  Simons Foundation Grant 506732.}

\begin{document}

\begin{abstract}
(i) Uncountably many synchronized reflected Brownian motions can hit the boundary of a $C^2$ domain at the same time. (ii) Measures associated to local times of two synchronized reflected Brownian motions are mutually singular until the time when the normal vectors at the reflection locations become identical. (iii) Mirror coupled reflected Brownian motions can simultaneously hit opposite sides of a wedge at different distances from the origin.  
\end{abstract}

\maketitle

\section{Introduction}

 We will prove three theorems on simultaneous hitting of the boundary by coupled Brownian motions.

The first theorem is essentially known, at least in a weaker form; see 
\cite{CLJ,Pi2005}. Consider a bounded $C^2$ domain $D\subset \R^d$, $d\geq 2$, and a stochastic flow of reflected Brownian motions starting from all points in $D$, driven by the same Brownian motion. Then, a.s., there is a time $t>0$ such that all processes that started from points in an open non-empty subset of $D$ are on  the boundary. Our contribution is a new proof based on  Brownian cone points.

The second theorem shows that the measures associated with local times of two reflected Brownian motions driven by the same Brownian motion are mutually singular before the time when the normal vectors at the reflection locations are identical.

The third and final theorem is our main result. It is concerned with ``mirror'' couplings, defined in Section \ref{d17.30}. These couplings were used many times to prove theorems in potential theory, see \cite{AB2,AB,AB3,BB1,BB2,B5}. The main arguments in all of these articles were based on the  analysis of the motion of the ``mirror,'' i.e., the line of symmetry for two coupled reflected Brownian motions. Mirror motion analysis is simple and intuitive as long as a certain simple construction (see Section \ref{d16.21}) of the mirror coupling can be applied. We will prove that, unfortunately, the simple construction is limited in its scope because two mirror coupled reflected Brownian motions can hit the sides of a wedge at the same time. 

\section{Synchronous couplings}

 Let $D\subset \R^d$ be a bounded connected open set with $C^2$-smooth boundary, for some $d\geq 2$. Let $\n(x)$ denote the unit inward
normal vector at $x\in\prt D$. Let $B$ be
standard $d$-dimensional Brownian motion, $x \in \ol D$, and consider the following Skorokhod equation,
\begin{align}\label{eq:j13.1}
 X^x_t &= x + B_t + \int_0^t  \n(X^x_s) dL^{x}_s,
 \qquad \hbox{for } t\geq 0.
\end{align}
Here $L^x$ is the local time of $X^x$ on $\prt D$. In other
words, $L^x$ is a non-decreasing continuous process which does
not increase when $X^x$ is in $D$, i.e., $\int_0^\infty
\bone_{D}(X^x_t) dL^x_t = 0$, a.s. Equation \eqref{eq:j13.1} has
a unique pathwise solution $(X^x,L^x)$ such that $X^x_t \in \ol D$
for all $t\geq 0$, simultaneously for all $x\in \ol D$ (see \cite{LS}). For every $x$, the reflected Brownian
motion $X^x$ is a strong Markov process.  We will
call the family $\{X^x\}_{x\in \ol D}$ a ``synchronous coupling.'' 
The construction of reflected paths in \cite{LS} is deterministic and dependence on initial conditions is continuous so the function $(x,t) \to X^x_t$ is jointly continuous, a.s.
Note that for any $x,y\in\ol D$ and any
interval $(s,t)$ such that $X^x_u \in D$ and $X^y_u \in D$ for all
$u \in (s,t)$, we have $X^x_u - X^y_u = X^x_s - X^y_s$ for all $u \in
(s,t)$.

Let $\B(x,r)$ denote the open ball with center $x$ and
radius $r$.

\begin{theorem}\label{prop:supp}
For every $y\in\ol D$, 
\begin{align*}
\P(\exists t,r>0\, \forall x\in \B(y,r): X^x_t\in \prt D) =1.
\end{align*}

\end{theorem}

\begin{proof}
For $\alpha \in (0,\pi)$, let 
\begin{align}\label{n30.1}
C(\alpha) &= \left\{(x_1,\dots,x_d) \in \R^d:
 x_1 > (\cot \alpha){\sqrt{x_2^2 +\dots + x_d^2}}\right\}.
\end{align}
We will say that $t>0$ is an $\alpha$-cone point for Brownian
motion $B$ if for all $s\in [0, t)$, we have $B_s \in B_t +
C(\alpha)$. It follows from the results in \cite{B1} that if
$\cos \alpha < 1/\sqrt{d}$ then cone points exist, a.s. Fix
some $\alpha \in( \arccos(1/\sqrt{d}), \pi/2)$. Standard
arguments based on Brownian scaling and the 0-1 law show that
with probability 1, for every $s>0$ there exists an
$\alpha$-cone point $t\in(0,s)$. Let $\bv =(1,0,\dots,0)$. By scaling, for every $p_1< 1$ there exists $c_1>0$ such that for every $r_1>0$, with probability greater than $p_1$, there exists an
$\alpha$-cone point $t$ such that $(B_t - B_0) \cdot \bv \leq- r_1$ and $|B_s-B_u|\leq c_1 r_1$ for all $s,u\in[0,t]$. Fix an arbitrary $p_1<1$ and the corresponding $c_1>0$. We will assume without loss of generality that $c_1>2$.

Since $D$ is a bounded domain with a $C^2$ boundary, there exists a point $w \in \prt D$ such that $\n(w) = \bv$.
It is easy to see that we can find $r_1>0$ so small that the following  conditions
are satisfied. 

(i) $\n(x)\cdot \bv \geq 1/2$ for  all $x \in \prt D \cap \B(w,4c_1 r_1)$. 

(ii) $|(x-w) \cdot \bv| \leq r_1/8$ for all $x\in \prt D\cap \B(w,4c_1 r_1)$.

(iii) If $x,z \in \R^d$, $|x-z|
\leq 8c_1r_1$, $x\in \prt D \cap \B(w,4c_1 r_1)$ and $x-z \in  C(\alpha)$ then $z \notin
\ol D$. 

(iv) If $(z-w) \cdot \bv \leq- r_1/2$ and $|z-w| \leq  2c_1 r_1$ then $z\notin \ol D$.

\medskip
Conditions (i)-(iv) are not logically independent but it is convenient to list them separately for reference.

Fix any $y\in \ol D$.
The process $X^y$ is neighborhood recurrent so 
the stopping time $T_1 := \inf\{t\geq 0: X^y_t \in \prt D \cap \B(w, r_1/8)\}$ is finite, a.s.
Since $(x,t) \to X^x_t$ is continuous, there exists $r>0$ such that 
\begin{align}\label{n25.1}
\P( \forall x\in \B(y,r): X^x_{T_1}\in \ol D\cap \B(w, r_1/4)) > p_1.
\end{align}
Let
\begin{align*}
A_1 & = \{\forall x\in \B(y,r): X^x_{T_1}\in \ol D\cap \B(w, r_1/4)\} ,\\
A_2 & = \{\text{there exists an $\alpha$-cone point $T_2$
for the process $\{B(T_1 +t), t\geq 0\}$}\\
&\qquad
\text{(that is, for all $s\in [T_1, T_2)$ we have $B_{s} \in B_{T_2} +
C(\alpha)$)}\\
&\qquad \text{such that 
$(B_{T_2} - B_{T_1}) \cdot \bv \leq- r_1$ and $|B_s-B_u|\leq c_1 r_1$ for all $s,u\in [T_1,T_2]$} \}.
\end{align*}
By the definition of $c_1$, the strong Markov property applied at $T_1$ and \eqref{n25.1},
$\P(A_1 \cap A_2)\geq p_1 ^2$. Since $p_1$ can be any number in $(0,1)$, it will suffice to show that if $A_1 \cap A_2$ occurred then $X^x_{T_2} \in \prt D$ for all $x\in \B(y,r)$.
Fix any $\omega \in A_1 \cap A_2$ and any $x\in \B(y,r)$.

First, we will show that $X^x_t \in \prt D$ for some $t\in[T_1,T_2]$. Suppose otherwise. Then $\int_{T_1}^{T_2} \n(X^x_s) dL^x_s=0$ and, therefore,
\begin{align}\label{n26.1}
X^x_{T_2} - w =  X^x_{T_2} - X^x_{T_1} + X^x_{T_1} - w =
B_{T_2} - B_{T_1} + X^x_{T_1} - w.
\end{align}
By the definitions of $A_1$ and $A_2$, and the assumption that $c_1>2$,
\begin{align}\label{n25.2}
|X^x_{T_2} - w| \leq
|B_{T_2} - B_{T_1}| +| X^x_{T_1} - w|
\leq c_1r_1 +r_1/4\leq (5/4) c_1 r_1 .
\end{align}
We use  \eqref{n26.1} and the definitions of $A_1$ and $A_2$ to see that
\begin{align}\label{n25.3}
(X^x_{T_2} - w)\cdot \bv &= 
(B_{T_2} - B_{T_1})\cdot \bv + (X^x_{T_1} - w)\cdot \bv
\leq -r_1 + | X^x_{T_1} - w| \\
&\leq -r_1+r_1/4 < -r_1/2.\notag
\end{align}
Condition (iv) applied to $z=X^x_{T_2}$ and \eqref{n25.2}-\eqref{n25.3} imply that $X^x_{T_2} \notin \ol D$, a contradiction. Hence, $X^x_t \in \prt D$ for some $t\in[T_1,T_2]$.

Next we will argue that $|X^x_u - w| \leq 4c_1r_1$ for $u\in [T_1, T_2]$. Suppose otherwise and let $T_3 = \inf\{t\geq T_1: |X^x_t - w| \geq 4c_1r_1\}\leq T_2$.
If $X^x_u \notin \prt D$ for $u\in[T_1,T_3]$ then
$|X^x_{T_3} - w| \leq (5/4)c_1r_1  $ because the argument proving \eqref{n25.2} remains valid if we replace $T_2$ with $T_3$. This contradicts the definition of $T_3$ so $T_4: = \sup \{t \leq T_3: X^x_t \in \prt D\}$ must exist and satisfy $T_4 \leq T_3\leq T_2$. 
By the definitions of $T_3,A_1$ and $A_2$,
\begin{align*}
\int^{T_4}_{T_1}&  dL^x_s
\geq\left|\int^{T_4}_{T_1} \n(X^x_s) dL^x_s\right|
=
\left|\int^{T_3}_{T_1} \n(X^x_s) dL^x_s\right|=
\left| X^x_{T_3} - X^x_{T_1} - B_{T_3} + B_{T_1}\right|\\
&\geq \left| X^x_{T_3}-w\right|-\left|w - X^x_{T_1}\right| - \left|B_{T_3} - B_{T_1}\right|
\geq 4 c_1 r_1 -r_1/4- c_1 r_1 \geq (5/2) c_1 r_1.
\end{align*}
This and condition (i) imply that
\begin{align*}
\left(\int^{T_4}_{T_1} \n(X^x_s) dL^x_s\right)\cdot \bv
\geq\frac12 \int^{T_4}_{T_1}  dL^x_s
\geq (5/4) c_1r_1.
\end{align*}
We use this bound, definitions of $A_1$ and $A_2$ and assumption that $c_1>2$ to see that
\begin{align*}
\left|(X^x_{T_4} - w) \cdot \bv\right|
&=\left|(B_{T_4} - B_{T_1}) \cdot \bv+
\left(\int^{T_4}_{T_1} \n(X^x_s) dL^x_s\right)\cdot \bv
+ (X^x_{T_1} - w) \cdot \bv\right|\\
&\geq
\left|\left(\int^{T_4}_{T_1} \n(X^x_s) dL^x_s\right)\cdot \bv\right|
-\left|(B_{T_4} - B_{T_1}) \cdot \bv\right|
-\left| (X^x_{T_1} - w) \cdot \bv\right|\\
&\geq (5/4) c_1r_1 - c_1r_1 - r_1/4 \geq  r_1/4.
\end{align*}
It follows from the definitions of $T_3$ and $T_4$ and condition (ii) that this is a contradiction. We conclude that $|X^x_u - w| \leq 4c_1r_1$ for $u\in [T_1, T_2]$
and, therefore, $|X^x_u - X^x_s| \leq 8c_1r_1$ for $s,u\in [T_1, T_2]$.

If $X^x_{T_2} \in \prt D$ then we are done. 
Suppose that $X^x_{T_2}
\notin \prt D$.
Recall that we have shown that $X^x_t \in \prt D$ for some $t\in[T_1,T_2]$.
Let $T_5 =\sup \{t< T_2: X^x_t \in \prt D\}$.
We have proved that $|X^x_u - w| \leq 4c_1r_1$ for $u\in [T_1, T_2]$
so $X^x_{T_5} \in \prt D \cap \B(w,4c_1 r_1)$.
The definition of an
$\alpha$-cone point implies that $B_{T_5} - B_{T_2} \in C(\alpha)$. Since
$\int_{T_5}^{T_2} \n(X^x_s) dL^x_s =0$, we obtain
\begin{align*}
X^x_{T_5} - X^x_{T_2}
= B_{T_5} - B_{T_2} - \int_{T_5}^{T_2} \n(X^x_s) dL^x_s \in C(\alpha).
\end{align*}
This, condition (iii) applied with $x=X^x_{T_5}$ and $z=X^x_{T_2}$, and the facts that  $X^x_{T_5} \in \prt D \cap \B(w,4c_1 r_1)$
and 
$|X^x_u - X^x_s| \leq 8c_1r_1$ for $s,u\in [T_1, T_2]$, imply
that $X^x_{T_2} \notin \ol D$, a contradiction. We conclude that
$X^x_{T_2} \in \prt D$. 
\end{proof}

Recall that $D$ is a $d$-dimensional $C^2$ domain, for some $d\geq 2$.
We define a measure $\mu_L^x$ on $[0,\infty)$ by $\mu_L^x( [s,t]) =
L^x_t - L^x_s$ for $t\geq s \geq 0$ and $x\in\ol D$.

\begin{theorem}\label{prop:sing}
Consider any $x,y\in \ol D$, $x\ne y$, and let 
\begin{align*}
T = \inf\{t \geq 0: X^x_t\in \prt D, 
X^y_t\in \prt D, \n(X^x_t) = \n(X^y_t)\}.
\end{align*}
Then, with probability 1, the measures $\mu_L^x$
and $\mu_L^y$ are mutually singular on $[0,T]$.
\end{theorem}

\begin{proof}
\emph{Step 1}.
Let $C_\bv(\alpha)$ be a cone with vertex 0, with the same angle as that of
$C(\alpha)$ defined in \eqref{n30.1}, and such that its axis contains $\bv$, a non-zero vector in
$\R^d$. 
Let $\angle(\bv,\bw) $ denote the angle between vectors $\bv$ and $ \bw$.
For non-zero vectors $\bv$ and $ \bw$, let
$\Lambda( \alpha,\bv,\bw)$ denote the set of times $t_3\geq 0$ such
that for some $t_4< t_3$ we have
$B_s \in B_{t_3} + \left( C_\bv(\alpha) \cap C_\bw(\alpha)
\right)$ for all $s\in [t_4, t_3)$. 

Suppose that for some $t_1$, $X^x_{t_1} \in \prt D$ and $X^y_{t_1} \in \prt D$. We will show that $t_1 \in \Lambda( \alpha,\n(X^x_{t_1}),\n(X^y_{t_1}))$ for every $\alpha > \pi/2$. Let $\bv =
\n(X^x_{t_1})$. Elementary geometry
shows that for every $\alpha > \pi/2$ there exist $\eps_1,\delta
>0$ such that if $ z\in \ol D$, $|z- X^x_{t_1}| \leq \eps_1$ and
$ \bw \in C_\bv(\delta)$ then $z + \bw \in X^x_{t_1} +
C_\bv(\alpha)$. Fix some $\alpha > \pi/2$ and corresponding
$\eps_1$ and $\delta$. Find $\eps_2 \in (0,\eps_1)$ so small
that $\n(z) \in C_\bv(\delta)$ for all $z\in \prt D$ such that
$|z- X^x_{t_1}| \leq \eps_2$. Let $t_2< t_1$ be such that $|X^x_s-
X^x_{t_1}| \leq \eps_2$ for all $s\in [t_2, t_1]$. Then
\begin{align*}
\bz_s := \int_{s}^{t_1} \n(X^x_u) dL^x_u \in C_\bv(\delta),
\end{align*}
for all $s\in [t_2, t_1)$. This implies that $X^x_{s} + \bz_s \in
X^x_{t_1} + C_\bv(\alpha)$, and, therefore,
\begin{align*}
B_s - B_{t_1}
= X^x_{s}  - X^x_{t_1} +\int_{s}^{t_1} \n(X^x_u) dL^x_u 
= X^x_{s} + \bz_s - X^x_{t_1} \in C_\bv(\alpha).
\end{align*}
Hence, $B_s \in B_{t_1} + C_\bv(\alpha)$ for all $s\in [t_2,
t_1)$. The same argument applies to $X^y_{t_1}$, so we conclude
that for every $\alpha > \pi/2$, one can find $t_2 < t_1$ such
that $B_s \in B_{t_1} + C_{\n(X^x_{t_1})}(\alpha)$ and $B_s \in
B_{t_1} + C_{\n(X^y_{t_1})}(\alpha)$ for all $s\in [t_2, t_1)$.
It follows that
$t_1\in\Lambda( \alpha,\n(X^x_{t_1}),\n(X^y_{t_1}))$.
Next, we will estimate the Hausdorff dimension of all times
$t_1$ with this property.

\medskip
\noindent
\emph{Step 2}.
The formula for the
Hausdorff dimension of ``cone points'' for planar Brownian
motion, derived in \cite{E}, is based on the tail properties of
the distribution of the exit time from the cone (see especially
Corollary 5 of \cite{E}). 

In the following, a ``cone'' is understood in the generalized sense, that is, any set $F\subset \R^d$ will be called a cone if $az\in F$ assuming that $z\in F$ and $a >0$.
The rate of decay of the tail of the
exit distribution is determined by the first Dirichlet
(spherical) Laplacian eigenvalue for the intersection of the
cone with the unit sphere (see, for example, Section 1 of
\cite{BB}). Hence, the Hausdorff dimension of
$\Lambda(\alpha,\bv,\bw)$ is determined by the first Dirichlet
(spherical) Laplacian eigenvalue for the intersection of the
cone $C_\bv(\alpha) \cap C_\bw(\alpha)$ with the unit sphere.
By an argument similar to the proof of Theorem 1.2 of \cite{BB}, when $\alpha \downarrow \pi/2$,
the eigenvalues corresponding to $C_\bv(\alpha) \cap
C_\bw(\alpha)$ converge to $C_\bv(\pi/2) \cap C_\bw(\pi/2)$.
The  asymptotic rate of decay for the tail of the exit time from 
$C_\bv(\pi/2) \cap C_\bw(\pi/2)$ is the
same as for the two-dimensional cone with angle $\pi- \angle(\bv,\bw)$.
Hence, as $\alpha \downarrow \pi/2$, the Hausdorff dimensions
of sets $\Lambda(\alpha,\bv,\bw)$ converge to  $1- \pi/(2 (\pi -
\angle(\bv,\bw)))$, by  arguments similar to those given in \cite{E}. 

For any unit vectors $\bv$ and $\bw$, let $\alpha>\pi/2$ be such that  the Hausdorff dimension
of  $\Lambda(\alpha,\bv,\bw)$ is less than  $1- \pi/(2 (\pi -
\angle(\bv,\bw)/2))$. Let 
$\alpha' = (\alpha + \pi/2)/2$, and let
$U(\bv,\bw)$ be the interior of the set of unit vectors $\bv'$ and $\bw'$ such that 
$C_{\bv'}(\alpha') \cap C_{\bw'}(\alpha') \subset C_\bv(\alpha) \cap C_\bw(\alpha)$. Note that $U(\bv,\bw)$ is open and non-empty.

The set of pairs of unit vectors $(\bv, \bw)$ such that $\angle(\bv,\bw) \geq 1/k$ is compact so it is covered by a finite family of sets $U(\bv,\bw)$. It follows that there exists $\beta(k) > \pi/2$ such that 
the Hausdorff dimension of $\Lambda^*_{k}:=\bigcup_{\angle(\bv,\bw) \geq 1/k} \Lambda(\beta(k),\bv,\bw)$ is less than $1- \pi/(2 (\pi -
1/(2k)))$.

\medskip
\noindent
\emph{Step 3}.
We have $\angle(\n(X^x_{t}),\n(X^y_{t}))>0$ for all $t< T$ so, by Step 1, 
\begin{align*}
 \{t\geq 0: X^x_{t} \in \prt D, X^y_{t} \in \prt D\}
\subset \bigcup_{k\geq 1} \Lambda^*_{k}.
\end{align*}
It will suffice to show that, for any fixed $k\geq 1$, neither $\mu^x_L$  nor $\mu^y_L$ charges $\Lambda^*_{k}$. Clearly, it is enough to supply a proof for $\mu^x_L$ only.

It has been shown in the proof of \cite[Thm. 3.2]{BCR}
that for every $\gamma<1/2$, the sample path 
of reflected Brownian motion in a smooth domain is $\gamma$-H\"older, a.s.
The same applies to Brownian motion paths so formula \eqref{eq:j13.1} implies that every component of the vector process
$t\to \int_0^t  \n(X^x_s) dL^{x}_s$ is $\gamma$-H\"older, a.s.
This in turn implies that $t\to  L^{x}_t$ is $\gamma$-H\"older, a.s.

Step 2 shows that the Hausdorff dimension of $\Lambda^*_{k}$, which we will call $\rho_k$, is strictly less than $1/2$. Consider $\gamma \in(\rho_k, 1/2)$, integer $m>0$ and a trajectory of $\{L^x_t, t\in[0, m]\}$ such that for some $c<\infty$, $|L^x_s - L^x_t| \leq c |t-s|^\gamma$ for all $s,t\in[0,m]$.
It follows from the definition of Hausdorff dimension that for every $\eps>0$ there exists a sequence of intervals $[s_j,t_j]$, $j\geq 1$, such that  $\Lambda^*_{k}\cap [0,m]
\subset \bigcup_{j\geq 1} [s_j,t_j]$ and $\sum_{j\geq 1} |t_j-s_j|^ \gamma < \eps$.
This implies that 
\begin{align*}
\mu_L^x(\Lambda^*_{k}\cap [0,m])
&\leq \mu_L^x \left( \bigcup_{j\geq 1} [s_j,t_j] \right)
\leq \sum_{j\geq 1} \mu_L^x ( [s_j,t_j])
= \sum_{j\geq 1} L^x _{t_j} - L^x_{s_j}\\
&\leq \sum_{j\geq 1} c |t_j-s_j|^\gamma
\leq c \eps.
\end{align*}
Since $\eps>0$ is arbitrarily small, $\mu_L^x(\Lambda^*_{k}\cap [0,m]) = 0$. Taking the sum over $m\geq 1$, we obtain $\mu_L^x(\Lambda^*_{k}) = 0$.
\end{proof}

\begin{remark}
Recall notation from Theorem \ref{prop:sing}.
If $D$ is a polygonal planar domain and $X^x_T$ and $X^y_T$ belong to the interior of the same edge of $\prt D$ then  for some random time $S>T$, $\mu^x_L([T,S])>0$ and measures $\mu^x_L$ and $\mu^y_L$ restricted to $[T,S]$ are identical.
\end{remark}

\section{Mirror couplings}\label{d17.30}

We will present three different constructions of ``mirror couplings'' of Brownian motions and reflected Brownian motions in planar domains, starting with couplings in the whole plane and then moving to  domains of greater complexity. These constructions were originally developed in
\cite{BK} and later applied in \cite{BB1} and other articles. Our review is similar to that in \cite{BB2}. 

\subsection{Mirror couplings in the plane}\label{d18.1}
 Suppose that $x,y\in \R^2$ are symmetric with respect to a
line $M$ and $x\ne y$. Let $X$ be a Brownian motion starting from $x$,
let $T^X_M = \inf\{t\geq 0: X\in M\}$,
 and let
$Y_t$ be the mirror image of $X_t$ with respect to $M$ for $t \leq
T^X_M$. We let $Y_t = X_t$ for $t>T^X_M$. By the strong Markov property applied at $T^X_M$, the process $Y$ is a
Brownian motion starting from $y$. The pair $(X,Y)$ is a ``mirror
coupling'' of Brownian motions in the plane.

\subsection{Mirror couplings in half-planes} \label{d16.21}
Informally speaking, a mirror coupling in a half-plane is the unique coupling of
reflected Brownian motions in the half-plane that behaves exactly
as the mirror coupling in the whole plane when both processes are
away from the boundary.  Suppose that $D_*$ is a
half-plane, $x, y \in D_*$, and let $M$ be the line of symmetry
for $x$ and $y$. The case when $M$ is parallel to $\prt D_*$ is
essentially a one-dimensional problem, so we focus on the case
when $M$ intersects $\prt D_*$. By performing rotation and
translation, if necessary, we may suppose that $D_*$ is the upper
half-plane and $M$ passes through the origin. We will write $x =
(r^x, \theta^x)$ and $y = (r^y, \theta^y)$ in polar coordinates.
The points $x$ and $y$ are at the same distance from the origin so
$r^x = r^y$. Suppose without loss of generality that $\theta^x <
\theta^y$. We first generate a 2-dimensional Bessel process $R_t$
starting from $r^x$. Then we generate two coupled one-dimensional
processes on the ``half-circle'' as follows. Let $\wt \Theta^x_t$
be a 1-dimensional Brownian motion starting from $\theta^x$. Let
$\wt \Theta^y_t = - \wt \Theta^x_t +\theta^x +\theta^y$. Let
$\Theta^x_t$ be reflected Brownian motion on $[0,\pi]$,
constructed from $\wt \Theta^x_t$ by the means of the Skorokhod
equation.  Thus $\Theta_t^x$ solves the stochastic differential
equation $d\Theta_t^x=d\wt \Theta_t^x+dL_t$, where $L_t$ is a
continuous process that changes only when $\Theta_t^x$ is equal to
$0$ or $\pi$ and $\Theta^x_t$ is always in the interval $[0,\pi]$.
The process $\Theta^x_t$ is constructed in such a way that the
difference $\Theta^x_t - \wt \Theta^x_t$ is constant on every
interval of time on which $\Theta^x_t$ does not hit $0$ or $\pi$.
The analogous reflected process obtained from $\wt \Theta^y_t$
will be denoted $\wh \Theta^y_t$. Let $\tau^\Theta$ be the
smallest $t$ with $\Theta^x_t = \wh \Theta^y_t$. Then we let
$\Theta^y_t = \wh \Theta^y_t$ for $t \leq \tau^\Theta$ and
$\Theta^y_t = \Theta^x_t$ for $t > \tau^\Theta$. We define a
``clock'' by $\sigma(t) = \int_0^t R^{-2}_s ds$. Then $X_t = (R_t,
\Theta^x_{\sigma(t)})$ and $Y_t = (R_t, \Theta^y_{\sigma(t)})$ are
reflected Brownian motions in $D_*$ with normal reflection---one
can prove this using the same ideas as in the discussion of the
skew-product decomposition for 2-dimensional Brownian motion
presented in \cite{IMK}. Moreover, $X$ and $Y$ behave like free
Brownian motions coupled by the mirror coupling as long as they
are both strictly inside $D_*$. The processes will stay together
after the first time they meet. We call $(X,Y)$ a ``mirror
coupling'' of reflected Brownian motions in half-plane.

The two processes $X$ and $Y$ in the upper half-plane remain at
the same distance from the origin. Suppose now that $D_*$ is an
arbitrary half-plane, and $x$ and $y$ belong to $D_*$. Let $M$ be
the line of symmetry for $x$ and $y$. Then an analogous
construction yields a pair of reflected Brownian motions starting
from $x$ and $y$ such that the distance from $X_t$ to $M \cap \prt
D_*$ is always the same as for $Y_t$. Let $M_t$ be the line of
symmetry for $X_t$ and $Y_t$. Note that $M_t$ may move, but only
in a continuous way, while the point $M_t \cap \prt D_*$ will
never move. We will call $M_t$ the {\it mirror} and the point $H :=
M_t \cap \prt D_*$ will be called the {\it hinge}. The absolute
value of the angle between the mirror and the normal vector to
$\prt D_*$ at $H$ can only decrease.

\subsection{Mirror couplings in polygons}\label{d16.20}
We will present an inductive construction of a mirror coupling
 $(X,Y)$ of reflected Brownian motions in a planar convex
polygonal domain $D$
based on the constructions presented in Sections \ref{d18.1} and \ref{d16.21}. We will construct a coupling only on
a (random) time interval $[0,S_\infty]$ such
that $X_t \notin \prt D$ or $Y_t
\notin \prt D$ for every $t\in[0,S_\infty)$. 

Assume that $x,y \in D$, $x\ne y$, and let $\{(X^1_t, Y^1_t), t\geq 0\}$ be the mirror coupling of Brownian motions in the whole plane, starting from $(X^1_0,Y^1_0) = (x,y)$.
Let $S_0=0$ and $S_1 =\inf\{t\geq 0: X^1_t \in \prt D \text{  or  } Y^1_t\in\prt D\}$.

If 
$X^1_{S_1} \in \prt D$ and $Y^1_{S_1} \in \prt D$ then we let $S_\infty=S_1$ and we end the induction. 

Suppose that either $X^1_{S_1} \notin \prt D$ or $Y^1_{S_1} \notin \prt D$. In the first case
let $I_1$ be the edge of $\prt D$ to which $Y^1_{S_1}$ belongs and let $K_1$
be the line containing $I_1$.
In the second case
let $I_1$ be the edge of $\prt D$ to which $X^1_{S_1}$ belongs and let $K_1$
be the line containing $I_1$.

Suppose that $\{(X^k_t, Y^k_t), t\geq S_{k-1}\}$, $S_k$, $I_k$ and $K_k$ have been defined and either $X^k_{S_k} \notin \prt D$ or $Y^k_{S_k} \notin \prt D$, for some $k\geq 1$.
Let $\{(X^{k+1}_t, Y^{k+1}_t), t\geq S_k\}$ be the mirror coupling of Brownian motions starting from $(X^{k+1}_{S_k}, Y^{k+1}_{S_k}) = (X^{k}_{S_k}, Y^{k}_{S_k})$,
constructed as in Section \ref{d16.21}, in the half-plane containing $D$, with boundary $K_k$.
Let $S_{k+1} =\inf\{t\geq S_k: X^{k+1}_t \in \prt D \text{  or  } Y^{k+1}_t\in\prt D\}$.

If 
$X^{k+1}_{S_{k+1}} \in \prt D$ and $Y^{k+1}_{S_{k+1}} \in \prt D$ then we let $S_\infty=S_{k+1}$ and we end the induction. 

Suppose that either $X^{k+1}_{S_{k+1}} \notin \prt D$ or $Y^{k+1}_{S_{k+1}} \notin \prt D$. In the first case
let $I_{k+1}$ be the edge of $\prt D$ to which $Y^{k+1}_{S_{k+1}}$ belongs and let $K_{k+1}$
be the line containing $I_{k+1}$.
In the second case
let $I_{k+1}$ be the edge of $\prt D$ to which $X^{k+1}_{S_{k+1}}$ belongs and let $K_{k+1}$
be the line containing $I_{k+1}$.

If there is no $k$ such that $S_\infty=S_k$ then we let $S_\infty = \lim_{k\to \infty} S_k$.

We define $(X_t,Y_t)$ for $t\in[0, S_\infty)$ by $(X_t,Y_t)=(X^k_t,Y^k_t)$
for $t\in[S_{k-1}, S_k)$ and  $k$ such that $S_{k-1} < S_\infty$. If $S_\infty < \infty$ then
we extend the definition of $(X_t,Y_t)$ to $t= S_\infty$ by continuity.

The construction of the mirror coupling can be easily continued beyond $S_\infty$ under some circumstances. For example, if $X_{S_\infty} = Y_{S_\infty}$ then $X$ and $Y$ can be continued beyond $S_\infty$ as a single reflected Brownian motion in $D$.

Let $M_t$ denote the mirror, i.e., the line of symmetry for $X_t$ and $Y_t$.
 Since the process which hits $I_k$ does not
``feel'' the shape of $\prt D$ except for the direction of $I_k$, it
follows that the two processes  remain at the same distance
from the hinge $H_t := M_t \cap K_k$ on the interval $[S_k, S_{k+1}]$. The mirror $M_t$ can move but
the hinge $H_t$  remains constant on the interval $[S_k, S_{k+1}]$. Typically, the hinge $H_t$
jumps at times $S_k$. The hinge $H_t$ may  lie
outside $\ol D$ at some times.

\subsection{Can mirror coupled reflected Brownian motions hit the boundary simultaneously?}
The
first rigorous construction of a mirror coupling  in a domain with
piecewise $C^2$-boundary was given in \cite{AB}. The construction given in \cite{AB} is rather technical so we find it of interest to determine whether the construction given in Section \ref{d16.20} can define a mirror coupling for all $t\geq 0$ in every convex polygonal domain. The positive answer would allow one to analyze the motion of the mirror using the elementary and intuitive methods  outlined in Section \ref{d16.21}. 
Our main result, given below, says that this is not possible.

\begin{remark}\label{d17.12}
Before we state our main result, we will list three possible situations when $X_{t}\in \prt D$ and $Y_{t}\in \prt D$. It is easy to see that each one of these
can occur with positive probability (for an appropriate domain and initial conditions). At the same time they do not pose any technical difficulties with the construction of the mirror coupling. Hence these three situations are not interesting.

(i) It may happen that $X_{t}=Y_{t}\in \prt D$ for some $t$.
 In this case, one can continue the mirror coupling as a single reflected Brownian motion in $D$ representing both $X$ and $Y$, after time $t$.

(ii) It may happen that $X$ and $Y$  hit the same edge $I$ at the same time $t$, at  different points. In this case the mirror is orthogonal to $I$ at time $t$.
One can easily continue the mirror coupling after time $t$, on some random time interval, until one of the processes hits a different edge of $\prt D$.

(iii)  If the mirror  passes through the intersection point of lines containing two edges $I$ and $J$  then it may happen that $X$ hits $I$ and $Y$ hits $J$ at the same time $t$. One can easily continue the mirror coupling after time $t$, on some random time interval, until one of the processes hits a different edge of $\prt D$.
\end{remark}

We will use complex and vector notation interchangeably.

\begin{theorem}\label{d17.11}
Consider a wedge
 $D = \{r e^{i\theta}\in \C: r>0,\ 0 < \theta < \alpha\}$  with angle $\alpha\in(0,\pi/2)$. We will denote the edges of $D$ by
 $E_X = (0,\infty)$ and $E_Y = \{r e^{i\alpha}: r> 0\}$. There exist $x,y\in D$ such that if $\{(X_t,Y_t), t\in[0, S_\infty)\}$ is the mirror coupling of reflected Brownian motions in $D$ constructed as in Section \ref{d16.20} and  $(X_0,Y_0) =(x,y)$ then 
\begin{align}\label{d17.10}
\P\left(S_\infty < \infty,X_{S_\infty}\in E_X, Y_{S_\infty}\in E_Y,
|X_{S_\infty}|\ne |Y_{S_\infty}| \right)>0.
\end{align}
\end{theorem}

\begin{remark}
(i) 
Recall that if $S_\infty <0$ then $X_{S_\infty}$ is defined as  $\lim_{t\uparrow S_\infty} X_t$. A similar remark applies to $Y_{S_\infty}$. 

(ii) It is easy to see that if the event in \eqref{d17.10} holds then none of the situations listed in Remark \ref{d17.12} (i)-(iii) could have occurred at time $S_\infty$.

\end{remark}

\begin{proof}[Proof of Theorem \ref{d17.11}]

\emph{Step 1}.
This step is devoted to a purely geometric lemma. We will investigate the effect of a change of one parameter in a geometric model on another parameter in the same model. 

Recall that $\angle$ denotes an angle; we will adopt the convention that all angles are in $[0,\pi]$.
For any points $F$ and $G$ in the plane, let $|FG|$ denote the distance between them. 
We will identify points in the plane with complex numbers and points on the real axis with real numbers.
Hence, if $F$ is a point in the positive part of the real axis then 
$F = |F| = |0F|$. Let $U:=\{x+iy\in \C: 0\leq y\leq \pi\}$.

 Suppose that $H>0$, $\beta > \alpha$ and $M= \{ H + r e^{i\beta}: r\in \R\}$. Define $H'$ by $\{H'\} = M \cap E_Y$. Let $\calS$ be the symmetry with respect to $M$, and define $A$ and $A'$ by $\{A\} = E_X \cap \calS(E_Y)$ and $\{A'\} = \{\calS(A)\} = E_Y \cap \calS(E_X)$. See Fig.~\ref{fig6}.

\begin{figure} \includegraphics[width=14cm]{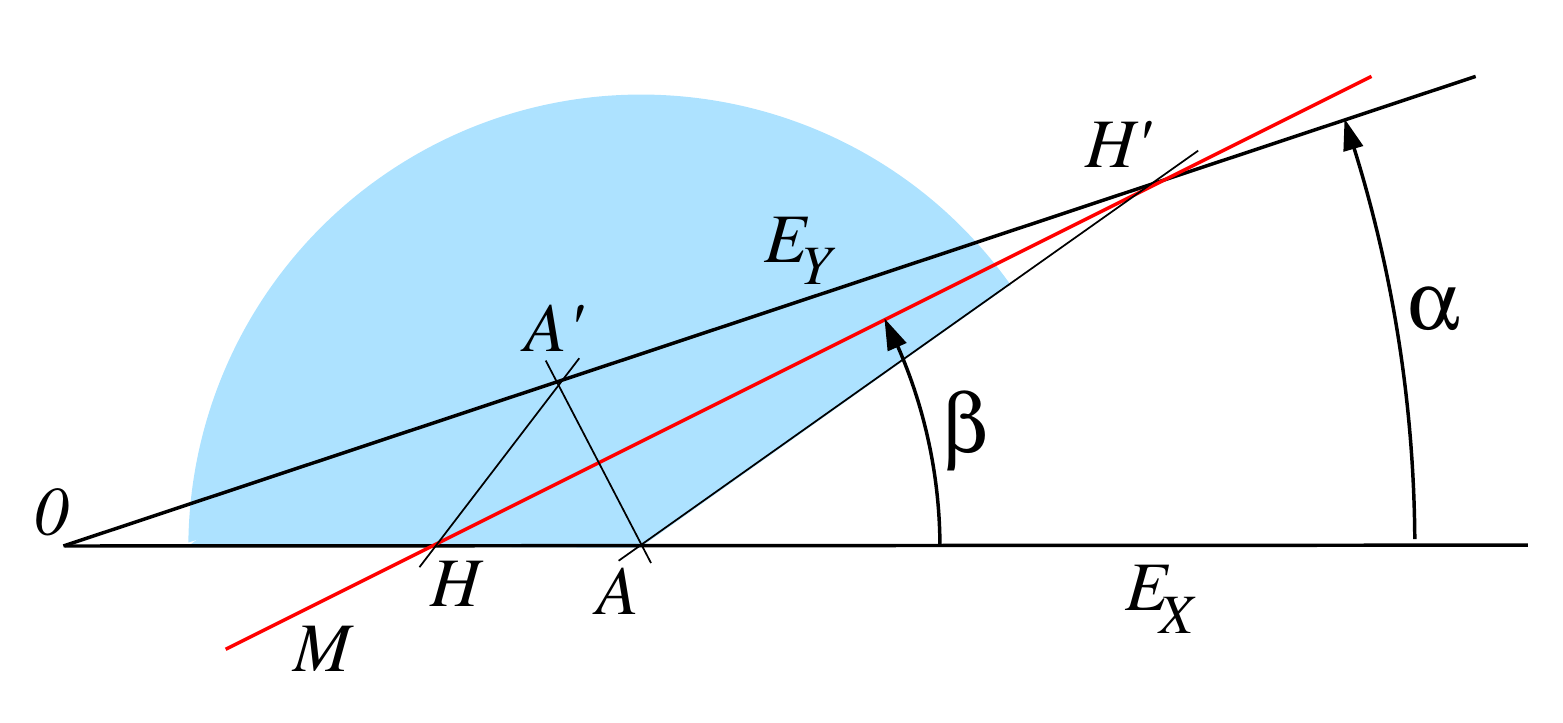}
\caption{ The first step of the proof is devoted to studying  the effect of increasing $\beta$ on the position of $A$ and the logarithmic transformation of the blue  wedge. 
}
\label{fig6}
\end{figure}

We will consider $\alpha$ and $H$ to be constants and we will treat $\beta$  as a variable. Note that $H', A $ and $A'$ are uniquely determined given $H, \alpha$ and $\beta$.

Elementary geometry shows that $\angle(0AH') = \pi + \alpha - 2 \beta$ and $\angle(0A'H) = 2\beta-\alpha$.
By the law of sines,
\begin{align}
&\frac{H}{\sin \angle(0A' H)}
= \frac{|HA'|}{\sin\angle(H0A')},\notag\\
&\frac{H}{\sin (2\beta-\alpha)}
= \frac{|HA'|}{\sin\alpha},\notag\\
&|HA|= |HA'| = \frac{H\sin\alpha}{\sin (2\beta-\alpha)},\notag\\
&A =  |H| + |HA| = H(1 + \sin\alpha\csc (2\beta-\alpha)).
\label{d9.1}
\end{align}

Let $W$ be the closed wedge with vertex $A$, such that its sides contain $H$ and $H'$, and $A'$ lies in its interior. 
Note that $\pi/4 +\alpha/2 < \pi/2$.
For $z\in W$ and $\beta\in(\alpha, \pi/4 +\alpha/2) $ let
\begin{align}\label{d17.2}
f(\beta,z)&= (\log (z- A) + i(\alpha - 2 \beta)) \frac{\pi}{\pi + \alpha - 2 \beta}\\
&=
\left(\log \left(z- H \left( 1+ \frac{\sin\alpha}{\sin (2\beta-\alpha)} \right)\right) + i(\alpha - 2 \beta)\right) \frac{\pi}{\pi + \alpha - 2 \beta}.\notag
\end{align}
The function $f(\beta,z)$ takes values in $U$ and, informally speaking, sends $(\beta,A)$ to $-\infty$.
Consider  $r\in(H,A)$. We use \eqref{d9.1} to see that
\begin{align*}
f(\beta,r)&= (\log (r- A) + i(\alpha - 2 \beta)) \frac{\pi}{\pi + \alpha - 2 \beta}\\
&= (\log (-r+ A) + i(\pi+\alpha - 2 \beta)) \frac{\pi}{\pi + \alpha - 2 \beta}\\
&=
\left(\log \left(-r+ H \left( 1+ \frac{\sin\alpha}{\sin (2\beta-\alpha)} \right)\right) \right) \frac{\pi}{\pi + \alpha - 2 \beta}
+i\pi.
\end{align*}
We use \eqref{d9.1} once again to get,
\begin{align}\notag
\frac{\prt}{\prt \beta} f(\beta,r)
={}& - \frac{2\pi H  \sin (\alpha) \cot (2 \beta-\alpha) \csc (2 \beta-\alpha)}
 {(\pi+\alpha-2 \beta )(H(1+ \sin (\alpha) \csc (2 \beta-\alpha))-r)}\\
& +\frac{2 \pi }{(\pi+\alpha-2 \beta )^2}
 \log (H(1+ \sin (\alpha) \csc (2 \beta-\alpha))-r)\notag\\
 ={}& - \frac{2\pi H  \sin (\alpha) \cot (2 \beta-\alpha) \csc (2 \beta-\alpha)}
 {(\pi+\alpha-2 \beta )(A-r)} +\frac{2 \pi }{(\pi+\alpha-2 \beta )^2}
 \log (A-r).\label{d17.5}
\end{align}
Recall that $\alpha\in(0,\pi/2)$ and fix $\beta^*_1$ and $\beta^*_2$ such that $\alpha< \beta^*_1 < \beta^*_2< \pi/2$ and $2\beta^*_2 - \alpha < \pi/2$. If $\beta\in[\beta^*_1, \beta^*_2]$ then $ \sin (\alpha) \cot (2 \beta-\alpha) \csc (2 \beta-\alpha) >0$. Hence we can find $c^*_1 = c^*_1(\alpha, H, \beta^*_1, \beta^*_2)
>0$ such that if $\beta\in[\beta^*_1, \beta^*_2]$ then
\begin{align}\label{d12.1}
\frac{\prt}{\prt \beta} f(\beta,r) < -c^*_1 (A-r)^{-1}.
\end{align}

Next we calculate the normal derivative of $f$ with respect to the second variable.
If we write $z= r e^{i\theta}$ then
\begin{align}\label{d12.2}
\left|
\frac{\prt}{\prt \theta} f(\beta,r e^{i\theta})\Big|_{\theta=0}
\right| =
 \frac{\pi}{\pi + \alpha - 2 \beta} (A-r)^{-1}.
\end{align}

We will derive an analogous estimate for a mapping corresponding to the other side of the wedge $D$.
Let $\gamma = \angle(0H'H)$ and note that $\gamma = \beta - \alpha$.
We will now consider $\alpha$ and $H'$ to be constants and we will treat $\gamma$  as a variable. Note that $H, A $ and $A'$ are uniquely determined given $H', \alpha$ and $\gamma$.

We have $\angle(0AH') = \pi - \alpha - 2 \gamma$. 
By the law of sines,
\begin{align}
&\frac{|H'|}{\sin \angle(0AH')}
= \frac{|H'A|}{\sin\angle(H'0A)},\notag\\
&\frac{|H'|}{\sin (\pi-\alpha - 2\gamma)}
= \frac{|H'A|}{\sin\alpha},\notag\\
&|H'A'|= |H'A| = \frac{|H'|\sin\alpha}{\sin (2\gamma+\alpha)},\notag\\
& |A'| = |H'| - |H'A'| = |H'|(1 - \sin\alpha\csc (2\gamma+\alpha)).
\label{d9.2}
\end{align}

Let $W'$ be the closed wedge with vertex $A'$, such that its sides contain $H$ and $H'$, and $A$ lies in its interior. 
Let $(v)^-$ denote the complex conjugate of $v\in\C$. For $z\in W'$ and $\gamma \in(0, \pi/4-\alpha/2)$ let 
\begin{align}\label{d17.3}
g(\gamma,z)&= 
\left((\log (z- A') - i\alpha ) \frac{\pi}{\pi - \alpha - 2 \gamma}
\right)^-\\
&=\left((\log (z- H'(1 - \sin\alpha\csc (2\gamma+\alpha))) - i\alpha ) \frac{\pi}{\pi - \alpha - 2 \gamma}
\right)^-.\notag
\end{align}
The function $g(\gamma,z)$ takes values in $U$ and, informally speaking, sends $(\gamma,A')$ to $-\infty$.
Consider $z= r e^{i\alpha}$ with $r\in(|A'|,|H'|)$. Using \eqref{d9.2}, we obtain
\begin{align*}
g(\gamma,z)&= 
\left((\log (r e^{i\alpha}- A') - i\alpha ) \frac{\pi}{\pi - \alpha - 2 \gamma}
\right)^-\\
&= 
\left(\log (r -| A'|)  \frac{\pi}{\pi - \alpha - 2 \gamma}
\right)^-\\
&= \log (r -| A'|)  \frac{\pi}{\pi - \alpha - 2 \gamma}\\
&=\log (r -|H'|(1 - \sin\alpha\csc (2\gamma+\alpha)))  \frac{\pi}{\pi - \alpha - 2 \gamma},
\end{align*}
so, using \eqref{d9.2} once again,
\begin{align*}
\frac{\prt}{\prt \gamma} g(\gamma,z)
={}&-\frac{2 \pi  |H'| \sin (\alpha) \cot (\alpha+2 \gamma) \csc (\alpha+2 \gamma)}{(\pi-\alpha-2 \gamma ) (r-|H'| (1-\sin (\alpha) \csc (\alpha+2 \gamma)))}\\
&+
\frac{2 \pi  \log (r-|H'| (1-\sin (\alpha) \csc (\alpha+2 \gamma)))}
{(\pi-\alpha-2 \gamma )^2}\\
&=
-\frac{2 \pi  |H'| \sin (\alpha) \cot (\alpha+2 \gamma) \csc (\alpha+2 \gamma)}{(\pi-\alpha-2 \gamma ) (r-|A'|)}
+\frac{2 \pi  \log (r-|A'|)}
{(\pi-\alpha-2 \gamma )^2}.
\end{align*}
Recall that $\alpha\in(0,\pi/2)$ and let $\gamma^*_1= \beta^*_1 -\alpha$ and $\gamma^*_2= \beta^*_2 -\alpha$. Then  
\begin{align}\label{d17.4}
0< \gamma^*_1 < \gamma^*_2< \pi/2-\alpha,
\qquad 2\gamma^*_2 + \alpha < \pi/2.
\end{align}
If $\gamma\in[\gamma^*_1, \gamma^*_2]$ then $ \sin (\alpha) \cot (2 \gamma+\alpha) \csc (2 \gamma+\alpha) >0$. Hence we can find $c^*_2 = c^*_2(\alpha, H', \gamma^*_1, \gamma^*_2)
>0$ such that if $\gamma\in[\gamma^*_1, \gamma^*_2]$ then
\begin{align}\label{d12.5}
\frac{\prt}{\prt \gamma} g(\gamma,r) < -c^*_2 (r-|A'|)^{-1}.
\end{align}

We will now calculate the normal derivative of $g$ with respect to the second variable.
Write $z= r e^{i\theta}$. Then
\begin{align}\label{d12.6}
\left| \frac{\prt}{\prt \theta} 
g(\gamma,r e^{i\theta})\Big|_{\theta=\alpha} \right| =
 \frac{\pi}{\pi - \alpha - 2 \gamma} (r-|A'|)^{-1}.
\end{align}

Recall that $\gamma = \beta - \alpha$ and
note that for fixed $\alpha,\beta$ and $H$, we have  for $z\in W'$,
\begin{align}\label{d17.1}
g(\gamma, z) = f(\beta, \calS(z)).
\end{align}

\medskip
\noindent
\emph{Step 2}.
Recall that $M_t$ denotes the line of symmetry for $X_t$ and $Y_t$, reflected Brownian motions in $D$. 
Assume that 
$M_0 = \{K+r e^{i \beta_0}: r\in \R\}$ for some $K \in E_X$ 
and $\alpha<\beta_0< \pi/4 +\alpha/2 $.

Let $\bar E_X$ ($\bar E_Y$) be the straight line containing $E_X$ ($E_Y$). 
Let $H_{X,t} = M_t \cap \bar E_X$ and $H_{Y,t} = M_t \cap \bar E_Y$.
Let $\calS_t$ be the symmetry with respect to $M_t$. In particular,
we  have $\calS_t(X_t) = Y_t$ for all $t$. Let $A_{X,t}$ and $A_{Y,t}$ be defined by $\{A_{X,t}\} = \bar E_X \cap \calS_t(\bar E_Y)$
and $\{A_{Y,t}\} = \{\calS_t(A_{X,t})\} = \bar E_Y \cap \calS_t(\bar E_X)$.
Our assumptions on $M_0$ and $\beta_0$ imply that $0 < H_{X,0} < A_{X,0}$ and $0 < |A_{Y,0}| < |H_{Y,0}|$.
See Fig. \ref{fig2}.

\begin{figure} \includegraphics[width=14cm]{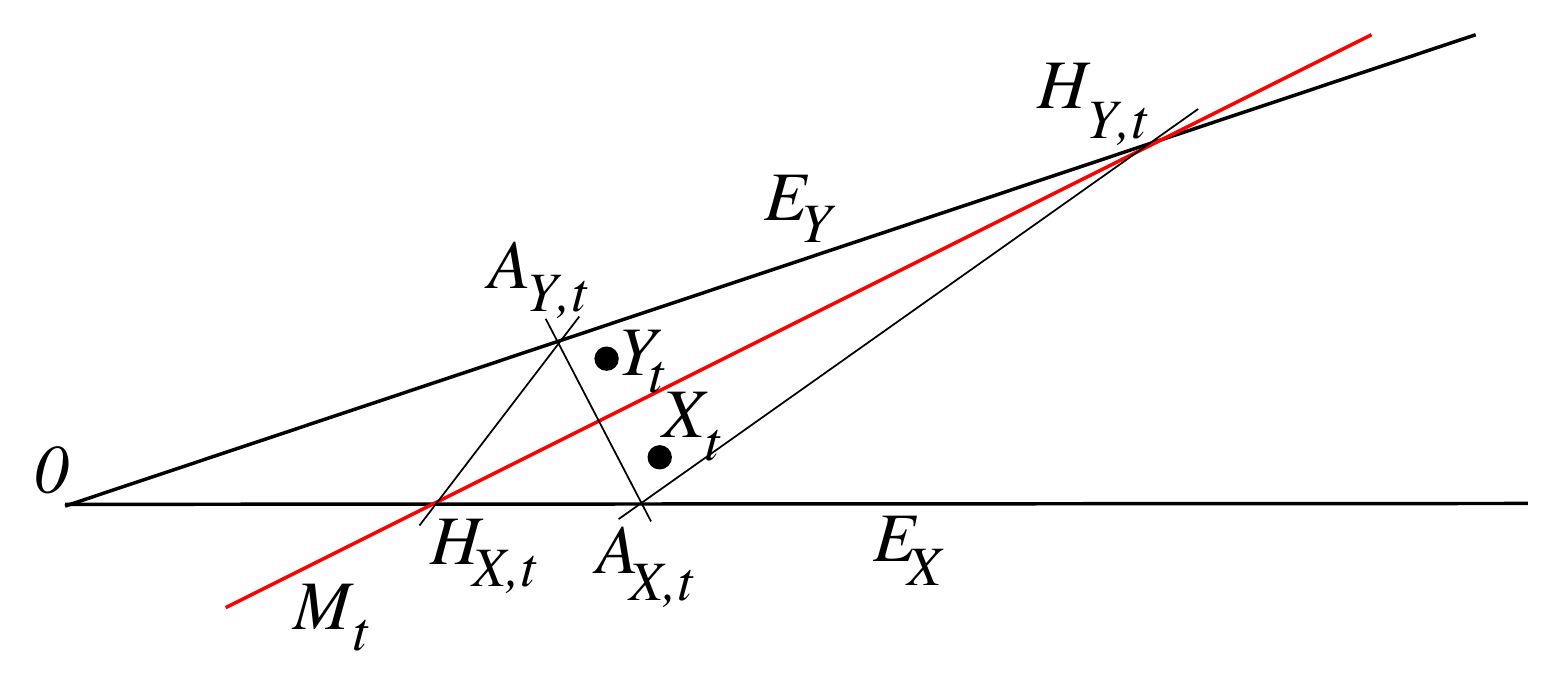}
\caption{ The proof shows that it is  possible for $X$ and $Y$ to visit the boundary simultaneously at a time $t$ such that  $X_t =A_{X,t}$, $Y_t = A_{Y,t}$ and $|X_t| \ne |Y_t|$.
}
\label{fig2}
\end{figure}

Let $\beta_t$ be defined by 
$M_t = \{H_{X,t}+r e^{i \beta_t}: r\in \R\}$ and 
\begin{align}\label{d16.11}
T'&= \inf \left\{t\geq 0:  X_t=Y_t \text{  or  } 0\in M_t 
\text{  or  } \beta_t \notin 
(\alpha, \pi/4 +\alpha/2)\right\}, \\
T''&= \inf \left\{t\geq 0: X_t \in \prt D \text{  and  } Y_t\in \prt D
\right\},\label{d16.8}\\
T & = T'\land T''. \label{d16.9}
\end{align}
The following definitions apply to $t\in[0,T)$.

Let $W_{X,t}$ be the closed wedge with vertex $A_{X,t}$, such that its sides contain $H_{X,t}$ and $H_{Y,t}$, and $A_{Y,t}$ lies in its interior. 
For $t\in[0,T)$ and $z\in W_{X,t}$ let 
\begin{align*}
\calF(t,z)&= (\log (z- A_{X,t}) + i(\alpha - 2 \beta_t)) \frac{\pi}{\pi + \alpha - 2 \beta_t}.
\end{align*}
Let $\gamma_t = \beta_t-\alpha$.
Let $W_{Y,t}$ be the closed wedge with vertex $A_{Y,t}$, such that its sides contain $H_{Y,t}$ and $H_{X,t}$, and $A_{X,t}$ lies in its interior. 
Recall that $(v)^-$ denotes the complex conjugate of $v\in\C$. For $t\in[0,T)$ and $z\in W_{Y,t}$ let 
\begin{align*}
\calG(t,z)&= 
\left((\log (z- A_{Y,t}) - i\alpha ) \frac{\pi}{\pi - \alpha - 2 \gamma_t}
\right)^-.
\end{align*}
It follows from \eqref{d17.2}, \eqref{d17.3} and  \eqref{d17.1} that 
\begin{align*}
\calF(t,z) &= f(\beta_t,z), \quad \text{  for  } t\in[0,T),z\in W_{X,t}, \\
\calG(t,z) &= g(\gamma_t,z),\quad \text{  for  } t\in[0,T),z\in W_{Y,t},\\
\calG(t, z) &= \calF(t,\calS(z)), \quad \text{  for  } z\in W_{Y,t},\\
\calG(t, Y_t)& = \calF(t,\calS(Y_t))= \calF(t,X_t), \quad \text{ for } t\in[0,T).
\end{align*}

The function $\calF(t,z)$ takes values in $U$ and sends $(t,A_{X,t})$ to $-\infty$.
The function $\calG(t,z)$ also takes values in $U$ and sends $(t,A_{Y,t})$ to $-\infty$.

If $X_t \in E_X$ for some $t$ then we will call both $X$ and $E_X$ \emph{active} at time $t$ (and similarly for $Y$ and $E_Y$).
Suppose that $X_t$ is active at time $t$. Then, over a short time interval $[t,t+\delta]$, the mirror $M_t$ will move from the position $M_t$ to $M_{t+\delta}$,  the angle $\beta_t$ will increase to $\beta_{t+\delta}$,  the wedge $W_{X,t}$ will be transformed into the wedge $W_{X,t+\delta}$, and the angle of $W_{X,t}$ will change from $\pi+\alpha - 2\beta_t$ to $\pi+\alpha - 2\beta_{t+\delta}$. As a result, $A_{X,t}$ will move to $A_{X,t+\delta}$ in the direction of $X_t$ (see Fig. \ref{fig4}). Analogous remarks apply to the situation when $Y$ is  active at time $t$ (see Fig. \ref{fig3}). For a point $z$ between $0$ and $A_{X,t+\delta}$,
its image under $\calF$ will change from $\calF(t,z)$ to $\calF(t+\delta,z)$. 

\begin{figure} \includegraphics[width=14cm]{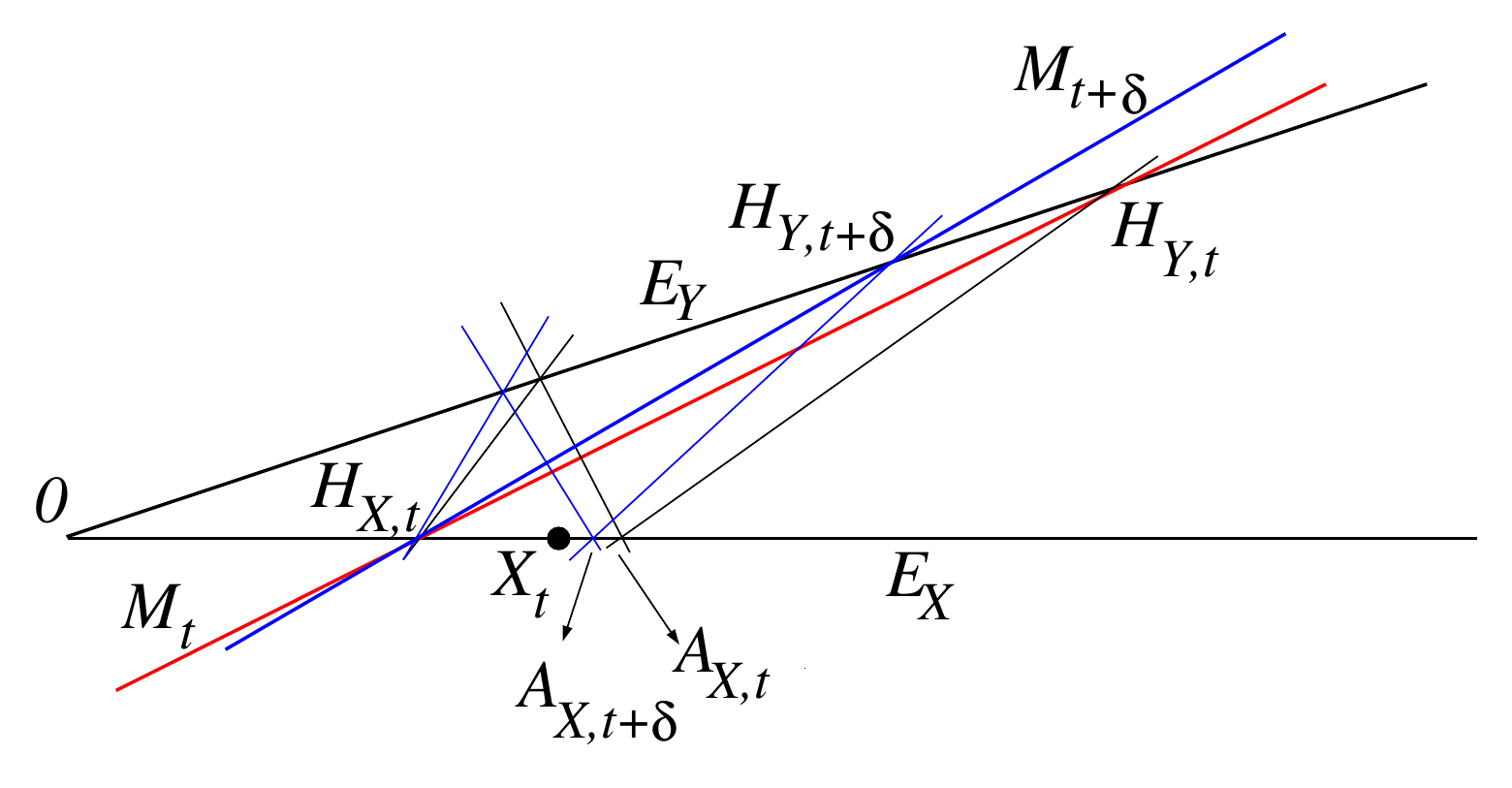}
\caption{If $X$ is active at time $t$ then  $A_{X,t}$ will move to $A_{X,t+\delta}$ in the direction of $X_t$.  
}
\label{fig4}
\end{figure}

\begin{figure} \includegraphics[width=14cm]{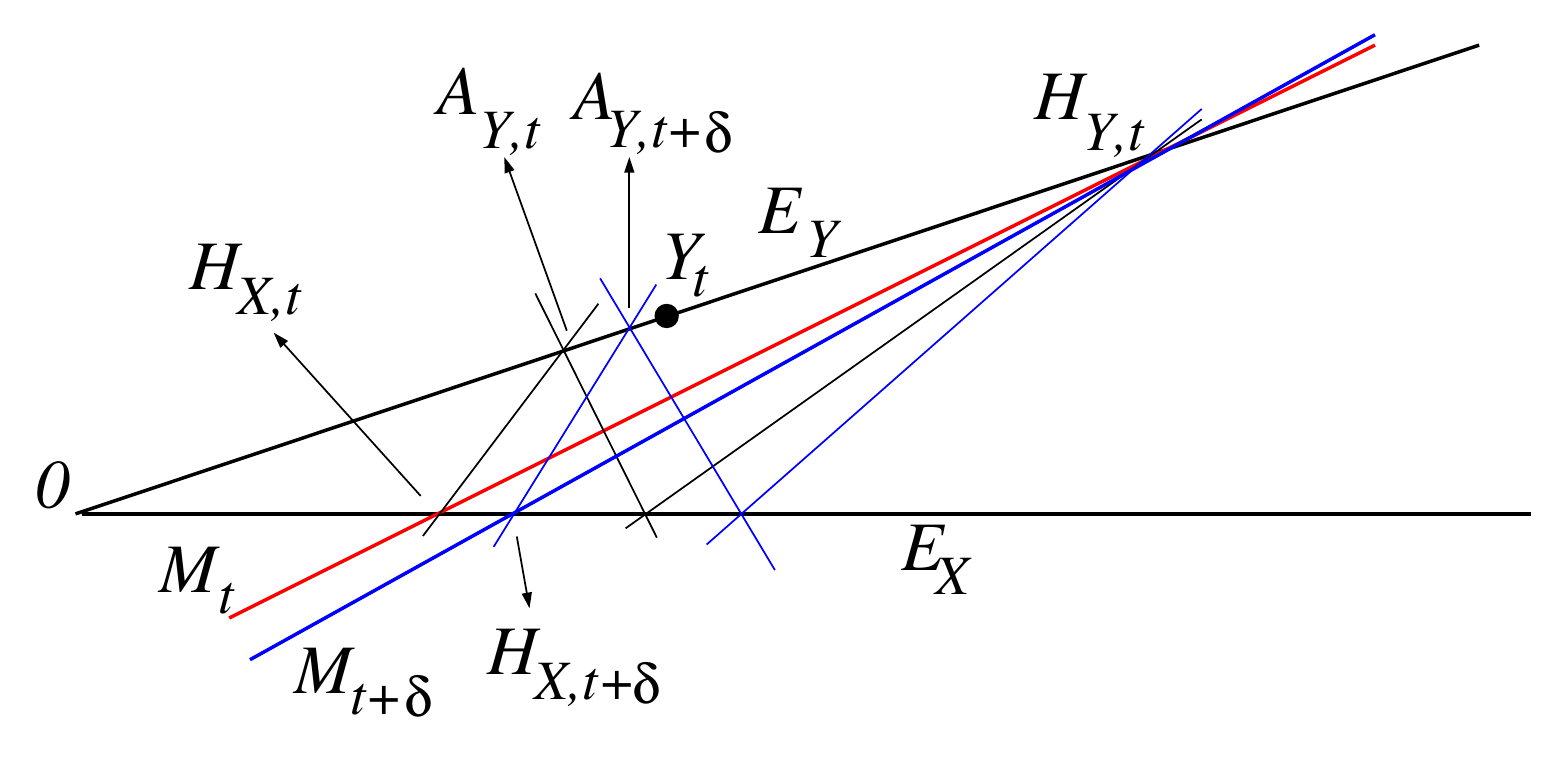}
\caption{If $Y$ is active at time $t$ then  $A_{Y,t}$ will move to $A_{Y,t+\delta}$ in the direction of $Y_t$.   
}
\label{fig3}
\end{figure}

Let
\begin{align}
Z^*_t&=\calG(t, Y_t) = \calF(t,X_t), \quad t\in[0,T),\notag\\
\wt \rho(t) &= \int_0^t 
\left| \left(\frac{d}{dz} \calF(s,z)\Big|_{z= X_s}\right) \right|^2 ds,
\quad t\in[0,T),\notag\\
s_* & = \lim_{t\uparrow T} \wt\rho(t),\label{d16.10}\\
\rho(t) &= \inf\{s\geq 0: \wt\rho(s) \geq t\},\notag\\
Z_s &= Z^*(\rho(s)),\quad s\in[0,s_*).\notag
\end{align}
The process $\{Z_t, t\in[0, s_*)\}$ is reflected Brownian motion in $U$ with (random) oblique reflection, by an argument very similar to the proof of  \cite[Thm. 2.3]{Pas}. We will not reproduce that proof here but we will point similarities. In \cite[Thm. 2.3]{Pas}, a reflected Brownian motion is transformed by a continuous   mapping depending on space and time. In that paper, there is a non-decreasing process, the norm of the original reflected Brownian motion, that is constant on time intervals whose union has full Lebesgue measure. On each of these intervals, the mapping does not depend on time and is analytic in the space variable.
In our case, the   local times $L^x_t$ and $L^y_t$ are constant on time intervals whose union has full Lebesgue measure. On each of these intervals, the mapping $\calF(t,z)$ does not depend on time and is analytic in the space variable.

The obliquely reflected Brownian motion $Z_t$ in $U$ has the following representation.
For some two-dimensional Brownian motion  $B'$ and $z_0=\calF(0,X_0) \in U$,
\begin{align}\label{d12.9}
 Z_t &= z_0 + B'_t + \int_0^t  \bv(s,Z_s) dL^{Z}_s,
 \qquad \hbox{for } t\in[ 0,s_*).
\end{align}
Here $L^Z$ is the local time of $Z$ on $\prt U$. In other
words, $L^Z$ is a non-decreasing continuous process which does
not increase when $Z$ is in the interior $U^\circ$ of $U$, i.e., $\int_0^{s_*}
\bone_{U^\circ}(Z_t) dL^Z_t = 0$, a.s. The vector of oblique reflection $\bv$ is normalized in \eqref{d12.9} so that the absolute value of its normal component is equal to 1.
The vector $\bv$ is random, i.e., $\bv(s, \,\cdot\,)$ depends on $\{Z_t, 0\leq t \leq s\}$.  We will write $\bv = (v_1, v_2) = v_1 + i v_2$, for $v_1,v_2\in\R$. Hence, $|v_2|=1$. More precisely, $v_2(z) = 1$ if $z\in \R$ and $v_2(z) = -1$ if $\Im(z) = \pi$. 

We will now determine the first component $v_1$ of the vector of oblique reflection $\bv$.
It follows from the construction of the mirror coupling in a half-plane
outlined in Section \ref{d16.21} that at the time when $X_t$ is active,  
\begin{align}\label{d15.1}
\Delta \beta_t = \frac{\Delta L_t^x}{|X_t-H_{X,t}|} ,
\end{align}
and, therefore,
\begin{align}\label{d12.3}
\frac{v_1(\rho(t),Z_{\rho(t)})}{|v_2(\rho(t),Z_{\rho(t)})|} = 
\frac{1}{|X_t-H_{X,t}|} \cdot
\frac{\frac{\prt}{\prt \beta} f(\beta,r)\Big|_{r=|X_t|,\beta=\beta_t}}
{\left|\frac{\prt}{\prt \theta} f(\beta,r e^{i\theta})\Big|_{\theta=0,r=|X_t|,\beta=\beta_t}\right|}.
\end{align}

Fix some $\beta^*_1$ and $\beta^*_2$ such that $\alpha< \beta^*_1 < \beta_0 <\beta^*_2< \pi/2$ and $2\beta^*_2 - \alpha < \pi/2$. Let
$T_1 = T \land \inf\{t\geq 0: \beta_t \notin [\beta^*_1, \beta^*_2]\}$. We combine \eqref{d12.1}, \eqref{d12.2} and \eqref{d12.3} to derive the following estimate for times $t\in[0,T_1)$ such that $X_t$ is active,
\begin{align}\label{d12.4}
\frac{v_1(\rho(t),Z_{\rho(t)})}{|v_2(\rho(t),Z_{\rho(t)})|} \leq
- \frac{c_1^*}{|X_t-H_{X,t}|} 
\cdot \frac{\pi + \alpha - 2 \beta_2^*}\pi 
=: - \frac{\wh c_1}{|X_t-H_{X,t}|}.
\end{align}
Note that $\wh c_1>0$.

Let $\gamma_1^* = \beta_1^* - \alpha$ and $\gamma_2^* = \beta_2^* - \alpha$.
Since $\gamma_t = \beta_t - \alpha$, we have $\gamma^*_1 < \gamma_0 < \gamma_2^*$. 
Note that
$T_1 = T \land \inf\{t\geq 0: \gamma_t \notin [\gamma^*_1, \gamma^*_2]\}$.
The conditions imposed on $\beta_1^*$ and $\beta_2^*$ imply that $0< \gamma^*_1 <\gamma_0< \gamma^*_2< \pi/2-\alpha$ and $2\gamma^*_2 + \alpha < \pi/2$. Hence, the assumptions  \eqref{d17.4} are satisfied. It follows that we can use \eqref{d12.5} and \eqref{d12.6} to derive the following estimate, analogous to \eqref{d12.4},  
for times $t\in[0,T_1)$ such that $Y_t$ is active,
\begin{align}\label{d12.7}
\frac{v_1(\rho(t),Z_{\rho(t)})}{|v_2(\rho(t),Z_{\rho(t)})|} \leq
- \frac{c_2^*}{|Y_t-H_{Y,t}|} 
\cdot \frac{\pi - \alpha - 2 \gamma_2^*}\pi 
=: - \frac{\wh c_2}{|Y_t-H_{Y,t}|},
\end{align}
where $\wh c_2>0$.

Let
\begin{align*}
r_1 = \frac 1 4 ( |A_{X,0}- A_{Y,0}| \land |A_{X,0}-H_{X,0}|
\land |A_{Y,0}-H_{Y,0}|).
\end{align*}
It follows from the construction of the mirror coupling given in Sections \ref{d16.21}-\ref{d16.20} that on the interval $[0, T_1]$, the distance $|0 H_{X,t}|$ is non-decreasing, the distance $|0 H_{Y,t}|$ is non-increasing, and the functions $\beta_t$ and $\gamma_t$ are non-decreasing. This  implies that there exist $\wh \beta_2\in (\beta_0, \beta^*_2)$ and $\wh \gamma_2\in (\gamma_0, \gamma^*_2)$  so small that 
if 
\begin{align}\label{d15.4}
T_ 2 = T\land \inf\{t\geq 0: \beta_t \notin [\beta_0, \wh\beta_2]\}\land \inf\{t\geq 0: \gamma_t \notin [\gamma_0, \wh\gamma_2]\}
\end{align}
then for $t\in [0, T_2]$,
\begin{align}\label{d18.2}
&|H_{X,t}-H_{X,0}| \leq r_1, 
\qquad |H_{Y,t}-H_{Y,0}| \leq r_1,\\
&|A_{X,t}-A_{X,0}| \leq r_1, 
\qquad |A_{Y,t}-A_{Y,0}| \leq r_1.\label{d18.3}
\end{align}

We will now argue that  
\begin{align}\label{d16.13}
\{T=T''= T_2\} \subset \{s_* = \infty\}.
\end{align}

It follows from \eqref{d17.5}, \eqref{d12.2} and \eqref{d12.3} that the vector of oblique reflection $\bv(\wt \rho(t), Z^*_t)$ is locally bounded on the upper boundary of $U$  for $t<T_2 \leq T$. The analogous remark applies to the lower boundary of $U$.

Suppose that $s_*<\infty$ and $T = T''=T_2$. We will show that this leads to a contradiction. The limit $\lim_{s\uparrow s_*} Z_s$  exists and is  finite  because $s_*$ is finite and $Z$ is a reflected Brownian motion with a locally bounded vector of oblique reflection, hence a continuous process. This implies that either $X_T\notin \prt D $ or $Y_T\notin \prt D$, hence $T\ne T''$, according to the definition \eqref{d16.8} of $T''$. This is a contradiction so we conclude that \eqref{d16.13} is true.

 Let $c_1 = (\wh c_1 \land \wh c_2) /r_1$.
It follows from \eqref{d12.4}, \eqref{d12.7} and \eqref{d18.2}-\eqref{d18.3} that we have the bound $v_1(\rho(t),Z_{\rho(t)}) \leq -c_1$ if $t\leq T_2$ and either $X_t\in \calB(A_{X,0},r_1)\cap \prt D $ or $Y_t\in \calB(A_{Y,0},r_1) \cap \prt D$.

It follows from \eqref{d15.1}, a formula analogous to \eqref{d15.1} for $\gamma_t$ (not stated explicitly), and \eqref{d18.2}-\eqref{d18.3} that there exists $c_2>0$ such that
\begin{align}\label{d16.5}
&\left\{L^x_t \leq c_2, L^y_t \leq c_2,
\sup_{0\leq s \leq t} |A_{X,0} - X_s| \leq r_1,
\sup_{0\leq s \leq t} |A_{Y,0} - Y_s| \leq r_1\right\}\\
&\subset
\left\{\beta_t \leq  \frac{\beta_0 +\wh\beta_2 }{2} ,
\gamma_t \leq  \frac{\gamma_0 +\wh\gamma_2 }{2}\right\}. \notag
\end{align}

Let $c_3>0$ be so small that if
\begin{align}\label{d16.6}
D_X &= \{w \in D: \dist(w, \prt D) \leq c_3, |w - A_{X,0}| \leq r_1\},\\
D_Y &= \{w \in D: \dist(w, \prt D) \leq c_3, |w - A_{Y,0}| \leq r_1\},
\label{d16.7} \\
\wh T_X &= \inf\{t \geq 0: X_t \notin D_X\},
\quad \wh T_Y= \inf\{t \geq 0: Y_t \notin D_Y\},\label{d16.4}\\
F_1& = \{L^x(\wh T_X) \leq c_2 \text{  and  }
L^y(\wh T_Y) \leq c_2\},\label{d16.1}
\end{align}
then 
\begin{align}\label{d16.2}
\P(F_1 \mid x  \in D_X,  y \in D_Y ) \geq 3/4.
\end{align}

Let $c_4\in \R$ be so small that 
\begin{align}\label{d15.2}
& \text{if $t\in[ 0,T_2]$  and $w\notin D_X$
then $\Re \calF(t, w) > c_4$, and} \\
&\text{if $t\in[ 0,T_2]$  and $w\notin D_Y$
then $\Re \calG(t, w) > c_4$.}\label{d15.3}
\end{align}

Assume for a moment that $s_*$ defined in \eqref{d16.10} is infinite, a.s. It is easy to see that no matter what oblique vector of reflection is, the local time on the boundary increases at a linear rate in the sense that for some $c_5>0$, $\lim_{t\to \infty} L^Z_t /t =c_5$, a.s. 
This easily implies that 
there exists $c_6 <c_4$ such that if
 the vector of oblique reflection $(v_1,v_2)$ satisfies $v_1(t,z) \leq -c_1$ for  $t\geq 0$ and  $z\in \prt U$ with $\Re z \leq c_4$, and
\begin{align}\label{d16.14}
F_2=\left\{\lim_{t\to \infty} \Re Z_t = -\infty, \ 
\sup_{t\geq 0} \Re Z_t < c_4\right\}
\end{align}
then 
\begin{align*}
\P(F_2\mid \Re Z_0 \leq c_6) \geq 3/4.
\end{align*}
It follows that if
\begin{align}\label{d16.15}
F_3=\left\{
\sup_{0 \leq t < \rho(T_2)\land s_*} \Re Z_t < c_4\right\}
\end{align}
then 
\begin{align}\label{d16.3}
\P(F_3\mid \Re Z_0 \leq c_6) \geq 3/4.
\end{align}
 
Recall that $(X_0,Y_0) = (x,y)$.
We choose $x,y\in D$ such that $ x  \in D_X$, $ y \in D_Y$ and $\Re\calF(0, x) = \Re \calG(0,y) \leq c_6 $.

The  event $F_1\cap F_3$ has probability greater than $1/2$ because of \eqref{d16.2} and \eqref{d16.3}.
Suppose that $F_1 \cap F_3$ occurred.

Assume that $\rho(T_2) < s_*$. We will show that this assumption leads to a contradiction.

Since $F_3$ occurred and $\rho(T_2) < s_*$, 
\begin{align}\label{d18.4}
\sup_{0 \leq t < \rho(T_2)\land s_*} \Re Z_t
= \sup_{0 \leq t < \rho(T_2)} \Re Z_t < c_4.
\end{align}
This, \eqref{d15.4}, \eqref{d16.4} and \eqref{d15.2}-\eqref{d15.3} imply that
\begin{align}\label{d18.5}
T_2 < \wh T_X \land \wh T_Y.
\end{align}
This, the assumption that $F_1$ occurred and
\eqref{d16.1}  imply that 
$L^x(T_2) \leq c_2$ and  $L^y(T_2) \leq c_2$.
Since
$T_2 < \wh T_X \land \wh T_Y$ holds, it follows from
 \eqref{d16.6}-\eqref{d16.7} that 
 $\sup_{0\leq s \leq T_2} |A_{X,0} - X_s| \leq r_1$ and
$\sup_{0\leq s \leq T_2} |A_{Y,0} - Y_s| \leq r_1$ hold.
This, coupled with the earlier observations, shows that the event on the left hand side of \eqref{d16.5} holds with $t$ replaced with $T_2$. Hence, the event on the right hand side of \eqref{d16.5} holds with $t$ replaced with $T_2$. But this contradicts the definitions of $T_2$ and $s_*$ and the assumption that $\rho(T_2) < s_*$. The proof that $\rho(T_2) \geq s_*$ is complete. 

The fact that $\rho(T_2) \geq s_*$
and the definitions of $s_*$ and $T_2$ given in \eqref{d16.10} and \eqref{d15.4} imply that $T = T_2 $
and, therefore, $\rho(T_2) = s_*$.
This, in turn implies that \eqref{d18.4} and \eqref{d18.5} remain valid.
It is elementary to check that $\wh T_X \land \wh T_Y < T'$.
Hence, the  fact that $T= T_2 < \wh T_X \land \wh T_Y$ implies that $T<T'$. Now it follows from \eqref{d16.9}  that $T=T''$. According to \eqref{d16.13}, $s_* = \infty$. This, the fact that $\rho(T_2) = s_*$ and the definition \eqref{d16.15} of $F_3$ imply that $\sup_{0 \leq t < \infty} \Re Z_t < c_4$. Comparing \eqref{d16.14} and \eqref{d16.15}, and recalling the discussion preceding \eqref{d16.14}, we conclude that $T< \infty$,
$ X_T =A_{X,T}\in E_X$  and  $Y_T=A_{Y,T}\in E_Y$. 
We have $|X_{T}|\ne |Y_T|$ because $T< T'$. The theorem holds with $S_\infty = T$.
\end{proof}

\section{Acknowledgments}

I am grateful to Zhenqing Chen for the most helpful advice.

\bibliographystyle{plain}
\bibliography{rbm}

\end{document}